\documentclass[a4paper,11pt]{amsart}

\pdfoutput=1

\usepackage[T1]{fontenc}
\usepackage[utf8]{inputenc}

\usepackage{amssymb,amsfonts,amsmath,amsthm}
\usepackage{url}
\usepackage{color}
\usepackage{enumerate}
\usepackage{epsfig,graphicx,psfrag}
\usepackage{tikz}
\usepackage{asymptote}
\usepackage{pinlabel}

\usepackage{mathabx}
\usepackage{bbm}
\usepackage{MnSymbol}

\usepackage[top=1.2in,bottom=1.4in,left=1.4in,right=1.4in]{geometry}

\input{xy}
\xyoption{all}
\objectmargin={3mm}

\newcounter{notes}
\newcommand{\ignore}[1]{}

\definecolor{maxime}{rgb}{0,.6,0}
\definecolor{fred}{rgb}{0,0,0.8}

\newtheorem{theorem}{Theorem}
\newtheorem{proposition}[theorem]{Proposition}
\newtheorem{corollary}[theorem]{Corollary}
\newtheorem{lemma}[theorem]{Lemma}

\newtheorem{observation}[theorem]{Observation}

\theoremstyle{definition}

\newtheorem{remark}[theorem]{Remark}

\newtheoremstyle{theoremwithref}{}{}{\itshape}{}{\bfseries}{.}{.5em}{#1 #2 #3}
\theoremstyle{theoremwithref}

\newcommand{\R}{\mathbb{R}}

\newcommand{\Z}{\mathbb{Z}}

\newcommand{\T}{\mathbb{T}}
\newcommand{\SL}{\mathrm{SL}}

\newcommand{\Hom}{\mathrm{Hom}}

\newcommand{\Diff}{\mathrm{Diff}}

\newcommand{\rot}{\mathrm{rot}}

\newcommand{\dq}{/\!/}

\DeclareMathOperator{\Homeo}{Homeo}

\newcommand{\Cfin}[1]{\mathcal{C}_{#1}^\dagger}

\newcommand{\cA}{\mathcal{A}}

\title{A note on weak conjugacy for homeomorphisms of surfaces}

\author[F. Le Roux]{Fr\'ed\'eric Le Roux}
\address{Sorbonne Universit\'es, UPMC Univ.\ Paris 06,
Institut de Math\'ematiques de Jussieu-Paris Rive Gauche,
UMR 7586, CNRS, Univ. Paris Diderot, Sorbonne
Paris Cit\'e, 75005 Paris, France}
\email{frederic.le-roux@imj-prg.fr}

\author[A. Passeggi]{Alejandro R. Passeggi}
\address{CMAT, Facultad de Ciencias, Universidad de la Rep\'ublica, Uruguay}
\email{alepasseggi@gmail.com}

\author[M. Sambarino]{Martin P. Sambarino}
\address{CMAT, Facultad de Ciencias, Universidad de la Rep\'ublica, Uruguay}
\email{samba@cmat.edu.uy}

\author[M. Wolff]{Maxime Wolff}
\address{Institut de Math\'ematiques de Toulouse, UMR5219,
UPS, F-31062 Toulouse Cedex~9, France}
\email{maxime.wolff@math.univ-toulouse.fr}

\begin{document}

\numberwithin{theorem}{section}

\begin{abstract}
  We explore the relation of weak conjugacy in the group of
  homeomorphisms isotopic to the identity, for surfaces.
\vspace{0.2cm}

\end{abstract}

\maketitle

\sloppy

\section{Introduction}\label{sec:IntroWC}

\subsection{Weak conjugacy}\label{ssec:WC}
Given any topological group $G$, we say that $f,g\in G$ are {\em weakly
conjugate} if, for every continuous conjugacy invariant $i\colon G\to Y$,
taking values in any Hausdorff space $Y$, we have $i(f)=i(g)$.

In this article, we investigate the relation of weak conjugacy in the
group $\Homeo_0(\Sigma)$ of homeomorphisms isotopic to the identity of
a compact surface $\Sigma$, endowed with the topology of uniform convergence.
In other words, we investigate the problem of determining a complete set of
{\em continuous} conjugacy invariants on the group $\Homeo_0(\Sigma)$,
with a particular interest to the case when $\Sigma$ is the torus $\T^2$.

This seems a natural question for itself, and our motivation also originates from
the analogy with character varieties, as we will explain now.

When $G$ is the group $\Homeo_+(S^1)$ of orientation-preserving homeomorphisms
of the circle, every map $f\in G$
has a Poincar\'e rotation number, $\rot(f)\in\R/\Z$, which is rational if and only
if $f$ has periodic orbits. In any case, by shrinking some intervals, the
rotation of angle $\rot(f)$ appears as a limit of conjugates of $f$, and it follows
that the rotation number is a complete weak conjugacy invariant.
This was generalized to actions of arbitrary discrete groups on the circle by
orientation-preserving homeomorphisms, by Ghys' {\em bounded Euler class}~\cite{Ghys87}.
Building upon work of Ghys~\cite{Ghys87}
and Matsumoto~\cite{Matsumoto86}, Mann and the last author observed that
two group actions on the circle are weakly conjugate if and only if they have
the same Ghys' bounded Euler class (see~\cite[Proposition~2.12]{RigidGeom}),
and formalized the notion of character space $X(\Gamma,G)$ for any discrete
group $\Gamma$ and topological group $G$, as the quotient of
$\Hom(\Gamma,G)$ by weak conjugacy. In this setting,
maps with values in the space of weak conjugacy classes of the group $G$ will
play the role of trace functions in classical character varieties,
and that makes it worthwhile to understand this space for different groups~$G$.

When $G=\Homeo_0(\T^2)$, Misiurewicz and Ziemian~\cite{MZ1} proved that, if
$f\in G$ and if $\tilde{f}\in\Homeo(\R^2)$ is a lift of $f$,
the sequence $\left(\frac{1}{n}\tilde{f}^n([0,1]^2)\right)_{n\geqslant 0}$ converges
to a nonempty, convex, compact subset $\rho(\tilde{f})$ of $\R^2$ which, up to
integer translations, depends only on $f$. This {\em rotation set} captures
information of the rotational dynamics of these maps, and can be used to recover
much information about the maps, for instance about periodic orbits
(see Franks~\cite{Franks}) or topological entropy (see Llibre-McKay~\cite{LM}).

By work of Misiurewicz-Ziemian, the interior of the rotation set $\rho(f)$ varies
continuously with the map $f\in\Homeo_0(\T^2)$. The first problem we
tried to address was whether this was a full invariant of weak conjugacy. We answer
this question by the negative; see Corollary~\ref{cor:RotNotEnough} below.

\subsection{The global nature of weak conjugacy}\label{ssec:GlobalNature}
We may easily produce maps of surfaces with rich dynamics but which are weakly
conjugate to the identity.
\begin{remark}\label{rmk:Alexander}
  Let $\Sigma$ be any surface, $D\subset\Sigma$ be a closed disk embedded in
  the interior of $\Sigma$, and let $f$ be any map supported in $D$.
  Then $f$ is weakly conjugate to the identity map of~$\Sigma$.
\end{remark}
Indeed, we may consider a disk $D'$ containing $D$ in its interior,
and then proceed with an Alexander trick in $D'$, to conjugate such maps to maps
with arbitrarily small support, hence arbitrarily uniformly close to the identity.

It should be noted here that the remark above would fail in higher regularity,
or in the context of conservative maps.
For example, the topological entropy is a continuous invariant in
the $C^\infty$ topology (see~\cite{Newhouse}).
For conservative maps, barcodes provide a continuous invariant in $C^0$
topology, see~\cite{LSV} and especially section 1.1 for a discussion of weak
conjugacy in that context.
A related discussion in different regularities in dimension~1 appears
in~\cite{Eynard-Bontemps-Navas}.

On the other hand, if $G$ is a subgroup of $\Homeo_0(\Sigma)$ endowed with a
finer topology, e.g., if $G=\Diff_0^r(\Sigma)$ for some $r\geqslant 1$,
then the weak conjugacy class in $\Homeo_0(\Sigma)$ of any element $g\in G$
obviously defines a continuous conjugacy invariant on $G$.

If we consider the space $\Homeo_0(S^2)$ of homeomorphisms of the sphere $S^2$
in the homotopy class of the identity, the elements supported in a closed disk
form a dense set.
This can be seen by first approximating an arbitrary element $f$ by a diffeomorphism.
This diffeomorphism then has a fixed point, and can be perturbed to new homeomorphism
$h$ that equals the identity on a small disk, and hence, is supported in a closed
disk. Remark~\ref{rmk:Alexander} then applies, and proves the following.

\begin{observation}\label{obs:SphereTrivial}
  All maps in $\Homeo_0(S^2)$ are weakly conjugate.
\end{observation}

Thus for the sphere, every continuous conjugacy invariant
in $\Homeo_0(\Sigma)$ is trivial. This suggests that for surfaces of higher genus,
continuous conjugacy invariants should involve global aspects of the surface,
like rotational properties.

\subsection{Time one of flows}\label{ssec:Flows}
The first main result of this article is that even in higher genus,
the time one maps of flows of vector fields cannot be distinguished by continuous
conjugacy invariants.
\begin{theorem}\label{thm:t1flow}
  Let $\Sigma$ be a compact surface and let $f\colon\Sigma\to\Sigma$
  be the time~1 of the flow of any $C^1$-vector field on $\Sigma$.
  Then $f$ is weakly conjugate to the identity.
\end{theorem}
When $\Sigma$ is the torus, Theorem~\ref{thm:t1flow} concerns in particular
every rotation. This presents some contrast with the case of the circle, where
rotations of distinct angles are not weakly conjugate.

We may combine Theorem~\ref{thm:t1flow} with the idea of Remark~\ref{rmk:Alexander}
to produce yet other examples of maps weakly conjugate to the identity. For
example, if the time one map $f$ of a flow restricts to a finite order rotation
in some annulus $A$ embedded in $\Sigma$, and $D$ is a disk embedded in $A$ and
disjoint from (or equal to) its images by powers of $f$, and $g$ is any map
supported in the orbit of $D$, then the composition $h=g\circ f$ is weakly conjugate
to the identity, as we may conjugate $h$ by maps commuting with $f$ and which
shrink the orbit of $D$. This implies that $f$ is a limit of conjugates of $h$,
hence $f$ and $h$ are weakly conjugate, and $f$ is itself weakly conjugate to
the identity by Theorem~\ref{thm:t1flow}.
This encompasses the so-called {\em odometer maps} on surfaces.

Besides all these maps and those from Remark~\ref{rmk:Alexander}, the
torus ``pseudo-rotations'' constructed in~\cite{BCL}, which may have positive
entropy, are obtained as limits of conjugates of a given rotation, and thus are
all weakly conjugate to the identity.

\subsection{Asymptotic translation length and width of the rotation set}\label{ssec:LenVSRot}
All the results we have stated so far express that
the weak conjugacy class of the identity,
in $\Homeo_0(\Sigma)$, contains a wide variety of elements.
Now, in the opposite direction, we will discuss some known continuous conjugacy
invariants, and precise some relations among them. The discussion
will mostly restrict to the case of the torus.

Recently, Bowden, Hensel and Webb~\cite{BHW} have introduced the {\em fine graph
of curves} $\Cfin{}(\Sigma)$, on a closed surface of genus $g\geqslant 1$, as follows.
It is the graph, whose vertices are the essential curves embedded in $\Sigma$,
and two vertices are connected by an edge whenever the two curves are disjoint.
When $\Sigma$ is a torus, we also put an edge between two vertices if the two
curves admit one intersection point.
Bowden, Hensel and Webb have proved that this graph is
Gromov-hyperbolic and has infinite diameter; this enabled them to construct
quasi-morphisms on
$\Diff_0(\Sigma)$, thereby answering a long standing question of Burago, Ivanov
and Polterovich. The group $\Homeo_0(\Sigma)$ acts on this Gromov-hyperbolic
graph by isometries, and for any $f\in\Homeo_0(\Sigma)$, the limit
\[ |f| = \lim_{n\to\infty}\frac{d(x,f^n(x))}{n} \]
is well-defined and does not depend on the vertex $x$ of the fine graph.
It is the {\em asymptotic translation length} of the action of $f$ on the
fine graph. Bowden, Hensel, Mann, Militon and Webb~\cite{BHMMW} have proved
that the range of the map $|\cdot|$ equals $\R_{\geqslant 0}$, and that this
map is a continuous, conjugacy invariant map.
They also proved that for any $f\in\Homeo_0(\T^2)$, the asymptotic translation
length $|f|$ vanishes, if and only if the rotation set $\rho(f)$ has empty interior.
They also provided some sufficient conditions for $f$ to act
elliptically or parabolically on $\Cfin{}(\T^2)$; this was completed subsequently
by Guih\'eneuf and Militon~\cite{GM}.

In fact, the method of \cite{BHMMW} is partly quantitative, and here we discuss in more
detail, quantitatively, the relation between $|f|$ and the {\em essential width}
$EW(\rho(f))$.
Given any compact convex set $C$ of the plane, set $EW(C)$ to be the
infimum of the length of the orthogonal projection of $A\cdot C$ to
the $x$-axis, as $A$ ranges over $\SL_2(\Z)$. In other words, it is the
infimum of the horizontal width of $C$, after any change of basis of $\Z^2$.
The proof of~\cite[Lemma 5.1]{BHMMW} yields the following inequality, for all
$f\in\Homeo_0(\T^2)$:
\[ EW(\rho(f)) \geqslant |f|. \]
This is the main part of their proof of the implication
$\ring{\rho(f)}=\emptyset\Rightarrow |f|=0$.
Their proof of the converse implication also
enables to keep track of explicit estimates. Yet, these estimates are non
uniform, and we prove that there exist no uniform reverse estimate for the
inequality above, by providing elementary, explicit examples, as follows.
Let $V\colon\R^2\to\R^2$ be defined by $V(x,y)=(x,y+\phi(x))$, where
$\phi\colon\R\to[0,1]$ is a continuous $1$-periodic function such that $\phi(0)=0$
and $\phi(1/2)=1$; for instance we may choose $\phi(x)=\sin^2(\pi x)$. This
descends to a map $v\colon\T^2\to\T^2$ which rotates at varying speed along the
vertical direction. Similarly, define $h\colon\T^2\to\T^2$ by the formula
$H(x,y)=(x+\phi(y),y)$. The compositions of these maps have simple enough dynamics
to keep track of their iterations and compute explicitly their rotation sets; this
makes such maps a classical example for maps with interesting rotation sets
(see e.g. \cite[Chapter~1]{KwaPHD} for a comparable usage).
Here, for every $n\geqslant 1$ we set $f_n=v^n\circ h^n$, and
$g_n=(v\circ h)^n$.

\begin{theorem}\label{thm:RotVSLength}{\ }
  
  \begin{enumerate}
  \item
    For all $c>0$, there exists $m_c>0$ such that for every $f\in\Homeo_0(\T^2)$,
    if $EW(\rho(f))\leqslant c$ then $|f|\geqslant m_c EW(\rho(f))$.
  \item
    For every $n\geqslant 1$, we have
    $\rho(f_n)=\rho(g_n)=[0,n]^2\textrm{ mod }\Z^2$, but
    $|f_n|\leqslant 2$ while $|g_n| = n |g_1|$ with $|g_1|>0$.
  \end{enumerate}
\end{theorem}
We suggest in Figure~\ref{fig:EWvsLength} a schematic picture of the possible
values of $EW(\rho(f))$ and~$|f|$.
\begin{figure}[htb]
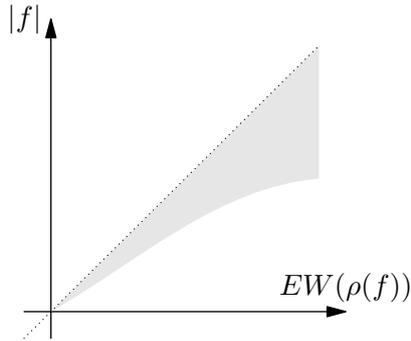

\begin{center}
\begin{asy}
  import geometry;
  
  path p = (0,0){dir(30)}..{dir(5)}(100,50)--(100,100)--cycle;
  
  draw ((-10,0)--(110,0), Arrow(7));
  draw ((0, -10)--(0, 110), Arrow(7));
  //label ("{\small $EW(\rho(f))$}", (0,110), W);
  //label ("{\small $|f|$}", (110,0), N);
  label ("{\small $|f|$}", (0,110), W);
  label ("{\small $EW(\rho(f))$}", (110,0), N);
  
  fill (p, lightgray);
  
  draw ((-10,-10)--(100,100), dotted);
\end{asy}
\end{center}
\caption{A schematic picture of the possible values of $EW(\rho(f))$ versus~$|f|$.}
\label{fig:EWvsLength}
\end{figure}
Here let us denote by $m_c$ the lowest possible constant for which statement (1) of
Theorem~\ref{thm:RotVSLength} holds. From the definition, the map $c\mapsto m_c$ is decreasing, and from the inequalities stated above, we know that $m_c\leqslant 1$
for all $c$, and that $m_n\leqslant \frac{2}{n}$ for all $n\geqslant 1$, in
particular $m_c$ goes to $0$ as $c$ goes to $+\infty$.
In Paragraph~\ref{ssec:Relations} we will see that
$m_c\geqslant\frac{1}{\max(888c, 1110)}$ for every $c>0$,
though this is obviously not sharp. This is all we know about
the map $c\mapsto m_c$, and Figure~\ref{fig:EWvsLength} is schematic.

Theorem~\ref{thm:RotVSLength}, (2), combined with the continuity of the
asymptotic translation length proved by Bowden, Hensel, Mann, Militon and Webb,
gives readily an answer to the question raised at the end of
Paragraph~\ref{ssec:WC} above.
\begin{corollary}\label{cor:RotNotEnough}
  There are maps with same rotation set and distinct weak conjugacy classes.
\end{corollary}
It is tempting however to ask whether maps with rotation sets of empty interior
are all weakly conjugate. 
In this direction, Militon has proved that every homeomorphism of the closed annulus whose rotation set is a single vector has conjugates arbitrarily close to the corresponding rotation (see~\cite[Theorem 1.5]{M18}). In the same paper, he conjectures that the same property holds for the torus and higher genus surfaces. This has been proved by Kwapisz in the torus case when the rotation vector is totally irrational (\cite{K03}).
Another basic open
problem concerns the group of homeomorphisms of the closed annulus that are the identity on the boundary: is there a single weak conjugacy class, or can one find some continuous conjugacy invariant on this group?

\subsection{Dynamical consequences}

The absence of uniform bounds in Theorem~\ref{thm:RotVSLength} has the following
easy consequence.
\begin{corollary}\label{cor:NoBloodyRoots}
  For every $\varepsilon>1$, there exists
  a non empty open set $U\subset G/\!/G$, with $G=\Homeo_0(\T^2)$, such that
  for every $[f]\in U$, the map $f$ has no $\frac{p}{q}$-th root,
  for any~$\frac{p}{q}\geqslant \varepsilon$.
\end{corollary}
In the statement above, we say that a map $h$ is a $\frac{p}{q}$-root of a
map $f$, when $h^p=f^q$.

Patrice Le Calvez informs us that it is possible to construct a horseshoe
homeomorphism having a $C^0$-neighborhood consisting of maps without any roots,
using Conley index theory to prove this last property.
On the other hand our method relies simply on the combination of two numerical
invariants (the two involved in Theorem~\ref{thm:RotVSLength}) and is very
robust, as this obstruction from having roots is stable under $C^0$ perturbation
as well as weak conjugacy.

All the homogeneous quasi-morphisms constructed in \cite{BHW}, as
highlighted there, are continuous conjugacy invariants.
These are most probably independent of both the rotation set and the translation
length. As a consequence of this continuity, together with Theorem~\ref{thm:t1flow},
we have the following.
\begin{corollary}
  The group $\Homeo_0(\Sigma)$ is not boundedly generated by time one maps of flows: 
  there does not exist an integer $N$ such that any
  element of $\Homeo_0(\Sigma)$  is the composition of $N$ time one maps of flows.
\end{corollary}
The proof is a very classical use of unbounded quasi-morphisms; we include it here.
\begin{proof}
  Let $\Phi\colon\Homeo_0(\Sigma)\to\R$ be any non trivial homogeneous quasi-morphism,
  as constructed in \cite{BHW}. As recalled there, $\Phi$ is continuous. Thus it
  vanishes identically on the weak conjugacy class of the identity.
  
  Let $h$ be such that $\Phi(h)\neq 0$.
  Then $\Phi(h^n)=n\Phi(h)$ goes to infinity when $n$ goes to infinity.
  On the other hand, assume that $h^n$ may be written as the composition
  of $N$ time one maps of flows $f_1,\ldots,f_N$.
  By Theorem~\ref{thm:t1flow}, $\Phi(f_i)=0$ for each $i$, and thus
  $\Phi(h^n)\leq (N-1)\Delta(\Phi)$,
  where $\Delta(\Phi)$ denotes the default of the quasi-morphism.
  This shows that there is no $N$ such that for every $n$, $h^n$ may be written
  as the composition of $N$ time one maps of flows.
\end{proof}

We would like to thank Kathryn Mann, Emmanuel Militon and Patrice Le Calvez for
encouraging and enlightening discussions.


\section{Proof of Theorem~\ref{thm:t1flow}}

\subsection{Some generalities regarding weak conjugacy}
We begin with a straightforward, useful exercise in general topology, which,
unfortunately, is not easy to locate in standard textbooks.
\begin{observation}
  For any topological space $X$, the intersection of all the Hausdorff
  equivalence relations on $X$ is a Hausdorff equivalence relation.
\end{observation}
In the statement above, we view an equivalence relation on $X$ as a
subset of $X\times X$; this gives sense to the intersection. Also, we
say an equivalence relation $\mathcal{R}$ on $X$ is Hausdorff, if the
quotient $X/\mathcal{R}$, endowed with the quotient topology, is Hausdorff.

The exercise above says that every topological space $X$ admits a smallest
Hausdorff equivalence relation. Equivalently, every topological space $X$
admits a biggest Hausdorff quotient $X/\sim$, which also satisfies the universal
property that any continuous map $f\colon X\to Y$ from $X$ to a Hausdorff
topological space $Y$ factors in a unique way as $f=\bar{f}\circ p$,
where $p\colon X\to X/\sim$ is the quotient and $\bar{f}$ is continuous.

If $G$ is a topological group, we may denote by $G/\!/G$ the biggest
Hausdorff quotient of the quotient $G/G$, where $G$ acts on itself by
conjugation. We may define that two elements $g,h\in G$ are weakly conjugate if
$p(g)=p(h)$, where $p\colon G\to G/\!/G$ denotes the projection map.
As follows from the lines above, this definition is equivalent to the one
given in the introduction.

It follows from the definitions that for any $g\in G$, its weak conjugacy class
\[ \left\lbrace h\in G, p(g)=p(h) \right\rbrace \]
is a closed subset of $G$, invariant by conjugation: this observation
has already been used implicitly in the introduction, and we will continue
to use it without mentioning it explicitly.

In the following we consider time one maps of flows associated to vector fields,
and use the notation where  $\phi^1_X$ is the time one map of the flows $\phi_X$ 
associated to vector fields $X$. We say that two flows are weakly conjugate
whenever their respective time one maps are.

\subsection{Definitions}\label{ss:definitions}
Recall that $X$ is a Morse-Smale vector field on a compact surface
(see Figure \ref{f.MorseSmale}) if:
\begin{itemize}
  \item the nonwandering set $\Omega(X)$ consists of a finite number of
    singularities and periodic orbits, all of them hyperbolic;
  \item there is no connection between saddle-type singularities.
\end{itemize}

\begin{figure}[h]
  \centering
  \def\svgwidth{0.75\columnwidth}
  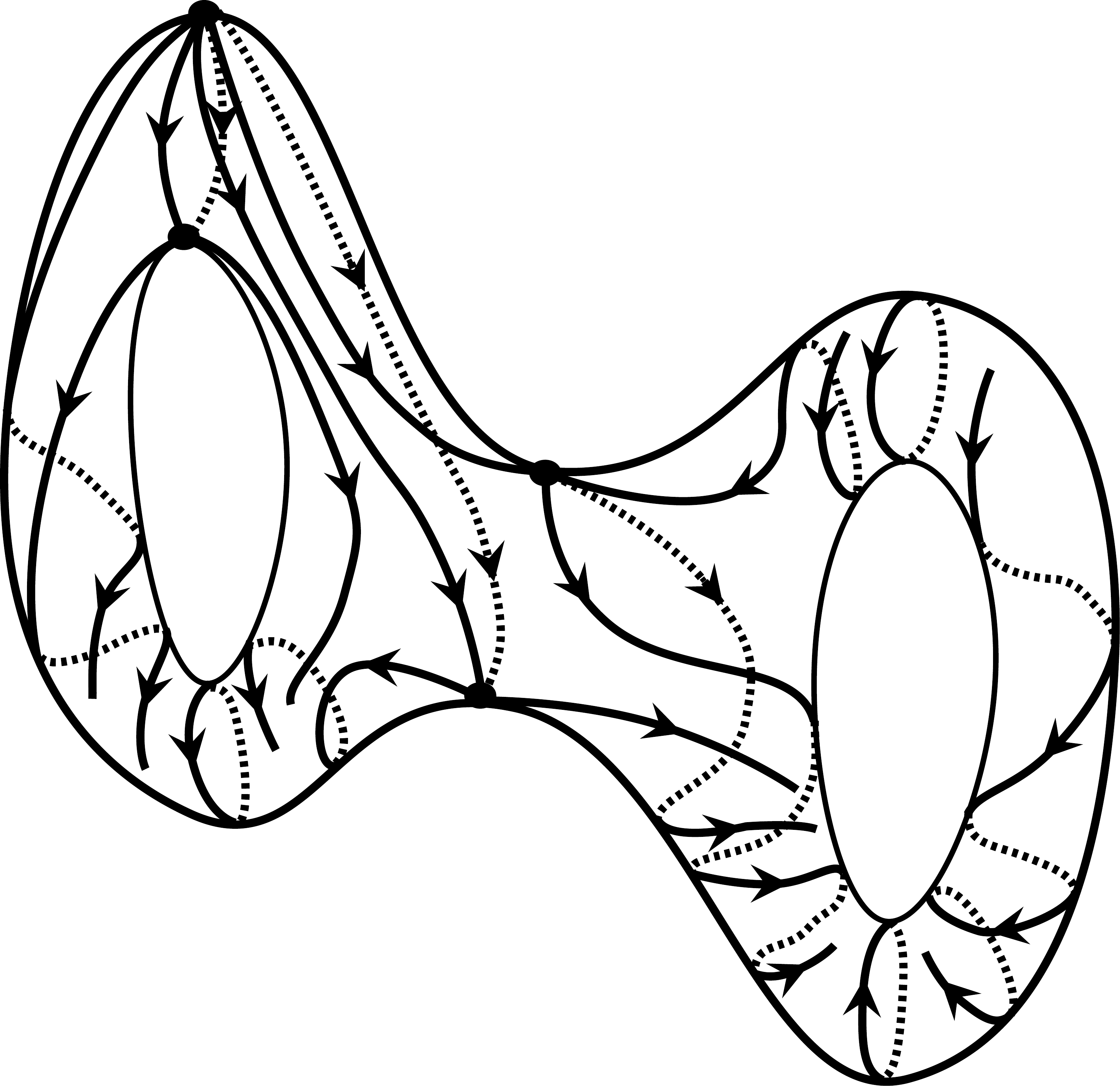
  \caption{Example of a Morse-Smale vector field on Surface of genus 2.
    It has four singularities (one repelling, three saddles) and three periodic
    orbits (one repelling and two attracting).}
  \label{f.MorseSmale}
\end{figure}

It is a classical result by M. Peixoto \cite{Peixoto} that the set
of Morse-Smale vector fields is $C^1$ (open and) dense inside the set of $C^1$ vector fields on a compact surface. Hence, to prove Theorem \ref{thm:t1flow} it is
enough to prove the following.

\begin{theorem}\label{t2flow}
  Let $X$ be a $C^1$ Morse-Smale vector field on a compact surface and
  let $\phi^1_X$ the time one map of flow generated by $X$. Then $\phi^1_X$ is
  weakly conjugate to the identity.
\end{theorem}

Working with Morse-Smale vector fields makes life much easier, since their
dynamics is fairly simple. More precisely, our proof will make use of two
crucial ingredients: (1) the existence of \emph{stable manifolds} for periodic
orbits, and (2) the existence of \emph{Conley circles}. We now recall both notions.

\subsubsection*{Stable manifolds} Let $O$ be an attracting hyperbolic periodic
orbit for the vector field $X$, and let $x$ be a point of $O$.
Then the \emph{local stable set} of $x$ is the image $W^s(x)$ of a $C^1$ embedded
arc $\gamma: [-1, 1] \to \Sigma$ such that $\gamma(0)=x$ and for
every $y \in W^s(x)$, $d(\Phi_{X}^t(y), \Phi_{X}^t(x))$ tends to zero when $t$
tends to $+\infty$. Furthermore, these stable local manifolds are positively
invariant under the flow, that is, for every $t>0$
we have $\Phi_{X}^t(W^s(x)) \subset W^s(\Phi_{X}^t(x))$.

\subsubsection*{Lyapounov functions}
K. Meyer~\cite{meyer} proved that if $X$ is a Morse-Smale vector field on a
manifold $M$ then there exists a $C^1$ map $\eta:M\to \R$ such that
\begin{itemize}
\item for any $x\in M$ and $t>0$ it holds
  that $\eta(\phi^t_X(x))\leqslant \eta(x)$ and equality holds if and only
  if $x$ is a singularity or a periodic orbit;
\item if $\beta_0,...,\beta_\ell$ denote the set of singularities and periodic
  orbits then $\eta(\beta_i)\neq\eta(\beta_j)$ if $i\neq j, 0\leqslant i,j \leqslant \ell$;
\item a point $x\in M$ is a critical point of $\eta$ if and only if $x$ is a
  singularity or a periodic orbit;
\item the vector field $X$ is transversal to the level sets of regular values of $\eta$.
\end{itemize}

\subsubsection*{Conley circles}
A {\em Conley circle} for a vector field $X$ is a $C^1$-submanifold
$C\subset\Sigma$, diffeomorphic to the circle, transverse to $X$, and
such that no orbit of the flow of $X$ meets $C$ more than once.
A finite union $C$ of pairwise disjoint Conley circles will be said {\em well
separating} if $\Sigma\smallsetminus C$ is the disjoint union of two open
sets $U^+$, $U^-$, such that for every $x\in C$, for all $t>0$,
$\phi_X^t(x)\in U^+$ and for all $t<0$, $\phi_X^t(x)\in U^-$.

If $\eta$ is a Lyapunov fonction for the vector field $X$, the level set of any
regular value of $\eta$ is a well-separating union of Conley circles.

\subsection{Outline of the proof}
Now we describe the main steps and ideas of the proof of
Theorem~\ref{t2flow}; the details will occupy the following subsections.
We start with a Morse-Smale vector field $X$.
Note that we do not leave the weak conjugacy class of $\phi_X^1$ by successively
conjugating and taking limits, and our strategy involves conjugations, limits,
as well as small perturbations of vector fields, until we reach the zero vector
field.

First, each family $C$ of Conley circles on $\Sigma$ whose orbits are pairwise
disjoint gives rise, through the flow of $X$, to a chart on $\Sigma$,
homeomorphic to $C\times \R$, on which the flow is vertical.
All flows of non-zero vector fields on $\R$ are conjugate, and similarly,
assuming the family of Conley circles is well separating,
we can conjugate $X$ to a multiple $s X$ with an appropriate
{\em slowdown function}, $s\colon\Sigma\to(0,1]$, defined
using this chart, taking small values near $C$, and $1$ outside a compact
subset of the chart. This conjugacy is first defined in the chart, and
coincides with some time of the flow,
$\Phi_X^{\tau_+}$, in the far future, and $\Phi_X^{\tau_-}$ in the past.
As a consequence, it can be extended outside of the chart image by
using $\Phi_X^{\tau_+}$ on $U^+$ and $\Phi_X^{\tau_-}$ on $U^-$, with the notation
above.

This first observation will enable us to slow down the flow of $X$ by
conjugacy, until, at the limit, that of a vector field $Y$ equal to zero
outside a regular neighborhood of the union of the periodic orbits.

These periodic orbits cannot be slowed down by conjugation, since the
periods are conjugacy invariants.
Still, the existence of the local stable
sets recalled above implies that the time one of the flow
of $Y$, on each remaining annulus, is conjugate to a very simple model. This will
imply that the time one $\Phi_Y^1$ is conjugate to arbitrarily
small perturbations of $\Phi_T^1$, where $T$ is a vector field which
has only periodic orbits in the annuli where $Y$ is nonzero
(the time one map of $T$ is a ``fibered rotation'').
Hence the times one of $T$ and $Y$ are weakly conjugate, and finally
there also exist arbitrarily small perturbations of $T$ which are
conjugate to arbitrarily small vector fields.

\subsection{Slowdown functions}\label{ssec:SlowDown}

Given a vector field $X$ on a manifold $\Sigma$, and a
diffeomorphism $f :\Sigma \to \Sigma'$, recall that $f_\ast X$ denotes the
image vector field on $\Sigma'$ defined by $f_\ast X(f(x)) = Df(x) X(x)$.
The flows of $X$ and of $f_\ast X$ are conjugate via the map $f$, that is, the
relation $\phi_{f_\ast X}^t(f(x)) = f(\phi_X^t(x))$ holds for every $x,t$.

\subsubsection*{Flows on $\R$} Let $X$ be a non-vanishing $C^1$ vector field
on $\R$. Then $X$ is conjugate to the constant vector field $\pm 1$: there exists
an orientation preserving diffeomorphism $f: \R \to \R$ such that $f_*X = \pm 1$.
Furthermore, such an $f$ is essentially unique: every other conjugacy
equals $f + t_0$ for some $t_0 \in \R$.

\subsubsection*{Slowdown functions on $\R$} Let $s: \R \to (0,1]$ be a $C^1$ function
which equals $1$ outside some compact subset of $\R$. We will call such a
function a {\em slowdown function} on $\R$.
Denote $X_s = sX$. According to the above paragraph, there exists a
diffeomorphism $f: \R \to \R$ such that $X_s = f_*X$. Furthermore, $X$ and $X_s$
coincide outside a compact subset $[\tau_-, \tau_+]$ of $\R$; as a consequence,
there exist two constants $t^-(s), t^+(s)$ such that the conjugacy $f$ coincides
with $\phi_X^{t^-(s)}$ on $(-\infty, \tau_-]$, and coincides
with $\phi_X^{t^+(s)}$ on $[\tau_+, +\infty)$.

\subsubsection*{Slowdown functions on $\Sigma$} Back to our surface $\Sigma$,
let $C$ be a union of pairwise disjoint Conley circles. Assume that the orbits
of any two Conley circles of the family are disjoint (this is the case
in particular if $C$ is well separating).
Then the map $\Psi: C \times \R \to \Sigma$ defined by the formula
$\Psi(c,t) = \phi_X^t(c)$ is injective, and defines a chart that sends the
unit vertical vector field $\mathbbm{1}(c, t) = (0,1)$ to our vector field $X$.
Given a slowdown function $s$ on $\R$, we consider the
function $s_{C, X}: \Sigma \to (0,1]$ defined by $s_{C,X}(\Psi(c,t)) = s(t)$ on
the image of the chart $\Psi$, and equal to $1$ elsewhere. We call this function
a slowdown function on $\Sigma$.
Let $f: \R \to \R$ be such that $f_\ast 1 = s$, and
let $F = \mathrm{Id} \times f: C \times \R \to C \times \R$, so
that $F_\ast \mathbbm{1} = (0, s(t))$. Then $\Psi$ sends this vector
field $(0, s(t))$ to the vector field $s_{C, X} X$.
On the image of the chart $\Psi$, the vector fields $X$ and $s_{C, X} X$ are
conjugate via the map $\Psi F \Psi^{-1}$. This map coincides with $\Phi_X^{t^-(s)}$
on $\Psi(C \times (-\infty, \tau_-])$, and coincides with $\Phi_X^{t^+(s)}$ on
$\Psi(C \times [\tau_+, +\infty))$.
If furthermore the family $C$ of Conley circles is well separating, then we can
define a global conjugacy between the vector fields $X$ and $s_{C, X} X$ by
extending the map $\Psi F \Psi^{-1}$ with $\Phi_X^{t^-(s)}$
on $U^- \smallsetminus \Psi(C \times \R)$
and with $\Phi_X^{t^+(s)}$ on $U^+ \smallsetminus \Psi(C \times \R)$.

\subsubsection*{Stopping functions} Passing to the limit, a map $s\colon\R\to[0,1]$
equal to $1$ outside some compact subset of $\R$ will be called a
{\em stopping function}. Such a function defines a {\em stopping function}
on $\Sigma$ by the same formula $s_{C,X}(\Psi(c,t)) = s(t)$ on the image of
the chart $\Psi$, and $1$ elsewhere.

The previous paragraph proves the following.

\begin{lemma}\label{lem:SlowDown}
  If $s$ is a slowdown function associated to a well separating family of
  Conley circles, then the flows of $X$ and $X_s = sX$ are conjugate.
  If $s$ is a stopping function associated to a well separating family of
  Conley circles, then the time one map
  of the vector field $X_s = sX$ is a limit of conjugates of the time 1 map
  of $X$; in particular, they are weakly conjugate.
\end{lemma}

\subsection{First step of the proof
of Theorem~\ref{thm:t1flow}}\label{ssec:DebutPreuveThm1}

Let $X$ be a Morse-Smale vector field on $\Sigma$.
Let us fix a Lyapunov function $\eta\colon\Sigma \to \R$ for $X$.
Consider any periodic orbit $O$ of $X$. Then the inverse image under $\eta$ of a
small interval around $\eta(O)$ has a connected component which is a small
annular neighborhood $A$ of $O$.
The boundary of $A$ is the union of two Conley circles which is
well separating. Applying Lemma~\ref{lem:SlowDown}, we get a vector field
$s X$ which is weakly conjugate to $X$, with $s$ a stopping function, which
vanishes exactly on a small regular neighborhood of the boundary of $A$.
It should be noted that the boundary components
of the set where $s$ vanishes are Conley circles of~$X$.

Let $O_1,\ldots,O_k$ denote the periodic orbits of $X$, $s_1,\ldots,s_k$ the
stopping functions, and $A_1,\ldots,A_k$ the annuli as obtained above.
We can arrange that the supports of the maps $1-s_i$ are disjoint.
Thus, the stopping steps as above can be applied
successively, each Conley circle of $X$ remaining a Conley circle until the step
when it is used. After successive applications of Lemma~\ref{lem:SlowDown}, the vector
field $Z= s_1\cdots s_k X$ is such that the time one maps $\phi_X^1$ and
$\phi_Z^1$ are weakly conjugate.
Let $Y$ denote the vector field that is equal to $Z$ in restriction
to $\Sigma_p=\cup_{i=1}^k A_k$, and $0$ on the rest of the surface.
\begin{lemma}\label{lem:Flow-Step1}
  The time one maps $\phi_Y^1$ and $\phi_X^1$ are weakly conjugate.
\end{lemma}
\begin{proof}
  Since $\phi_Z^1$ is weakly conjugate to $\phi_X^1$, it suffices to prove
  that $\phi_Y^1$ is weakly conjugate to $\phi_Z^1$.
  Let $\varepsilon>0$. It suffices to prove that $\phi_Z^1$ is conjugate to the
  time one map of a vector field which is equal to $Z$ on $\Sigma_p$ and has
  norm less than $\varepsilon$ on the complement $\Sigma'=\Sigma\smallsetminus\Sigma_p$.
  Consider the compact subset
  \[
  K_\varepsilon = \left\lbrace x\in\Sigma',\, ||Y(x)||\geqslant\varepsilon \right\rbrace
  \]
  of $\Sigma'$.
  Let $u_1,\ldots,u_n$ be a sequence of regular values of the Lyapunov function
  $\eta$ for $X$, which separates all its singular values, and denote
  $C_i=\eta^{-1}(u_i)$, for $i=1,\ldots,n$. Remember that each $C_i$ is a well   	  
  separating union of Conley circles. Note that
  $\cup_{i=1}^n \phi_X^\R(C_i)$ is the union of all non periodic and non singular
  orbits.
  
  The construction of the vector field $Z$ above took place in an arbitrarily small
  neighborhood of the periodic orbits, hence we may suppose that the function
  $s_1\cdots s_k$ is constant equal to $1$ on all the $C_i$'s.
  It follows that $\cup_{i=1}^n\phi_Z^\R(C_i)$ contains $K_\varepsilon$.
  By compactness of $K_\varepsilon$, there exists a positive real number $T$
  such that
  \[ K_\varepsilon \subset \bigcup_{i=1}^n \phi_Z^{[-T,T]}(C_i). \]
  
  Now the proof of Lemma~\ref{lem:Flow-Step1} proceeds by applying  $n$ times
  Lemma~\ref{lem:SlowDown}. Denote $M=\max_{x\in\Sigma'}(||Z(x)||)$.
  Put $Y_1=Z$. Consider a first slowdown function
  $s_{C_1, Y_1}$ which takes values in $(0,\varepsilon/M]$ on the set
  $\phi_{Y_1}^{[-T,T]}(C_1)$. By Lemma~\ref{lem:SlowDown}, the vector field
  $Y_2=s_{C_1, Y_1}\cdot Y_1$ is conjugate to $Y_1$.
  Let $\delta_1>0$ be the minimum of the map $s_{C_1,Y_1}$. The fields $Y_2$ and  		
  $Y_1$ are everywhere proportional, with coefficient no less than $\delta_1$. It 
  follows that for all $i=2,\ldots, n$ we have
  $\phi_Z^{[-T,T]}(C_i)\subset\phi_{Y_2}^{[-T/\delta_1,T/\delta_1]}(C_i)$.
  Now we can consider a second slowdown function $s_{C_2,Y_2}$, which takes
  values in $(0,\varepsilon/M]$ on the set $\phi_{Y_2}^{[-T/\delta_1,T/\delta_1]}(C_2)$,
  then set $Y_3 = s_{C_2,Y_2}\cdot Y_2$, and so on. The vector field $Y_{n+1}$
  obtained after $n$ steps of this construction, has norm everywhere less than
  $\varepsilon$ on the complement of $\Sigma_p$, and its time one map is conjugate
  to that of $Z$. This proves the lemma.
\end{proof}

\subsection{Topological model for attractive periodic orbits}
We consider the annuli $A_i$ as in the preceding section, and further denote
by $A_i'$ the closure of the subset of $A_i$ where $Y$ is nonzero.
In this section we describe the conjugacy class of the flow on each closed annulus $A_i'$.
We switch to the topological category which is more convenient for what follows.
Let us recall that a \emph{continuous flow} on $\Sigma$ is a group homomorphism from $\R$ to $\mathrm{Homeo}(\Sigma)$ which is continuous for the topology of uniform convergence.

\bigskip

Let $r >0$, and $\cA^+(r)$ be the set of continuous flows $(\Phi^t)_{t \in \R}$ on the compact annulus $\mathbb{A} = \mathbb{S}^1 \times [-1,1]$ which satisfy the following properties.

\begin{enumerate}
\item  Every point on the boundary is a fixed point of the flow.
\item The flow has a unique periodic orbit $O$, attracting, with period $r$,
  which is homotopic to $\mathbb{S}^1 \times \{0\}$.
\item There exists an embedded arc $\gamma : [-1, 1] \to \mathbb{A}$,
  positively invariant under $\Phi^r$, and which has with $O$ a unique,
  transverse intersection point $\gamma(0)$.
\item For every point $x$ interior to $\mathbb{A}$ and not in $O$,
  the $\omega$-limit set $\omega(x)$ is equal to $O$ and the $\alpha$-limit
  set $\alpha(x)$ is a single point, on the boundary of $\mathbb{A}$.
\item This defines a map $\alpha \colon \mathbb{A} \smallsetminus O \to \partial A$
  which is continuous, and induces a bijection between the set of non periodic
  orbits of $X$ and the boundary of the annulus.
\end{enumerate}

In the lemma below we identify $A_i'$ with the annulus $\mathbb{A}$, endowed with
the flow of~$Y$. We will suppose that the unique periodic orbit $O_i$ of $Y$
in $A_i$ is attracting, without loss of generality up to replacing $Y$ with $-Y$.
We denote $r_i$ the period of~$O_i$.
\begin{lemma}\label{lem:MSFitsModel}
  This flow is an element of $\cA^+(r_i)$.
\end{lemma}
\begin{proof}
  Remember that the boundary components of $A_i'$ are Conley circles for $X$, and
  that in a neighborhood of both boundary circles, the flow of $Y$ is conjugate to the
  product of the identity on circle and a flow on $[0, 1)$ which vanishes only
  at $0$. It follows that~(4) holds.
  We now discuss the existence of the invariant arc $\gamma$ in point (3).
  Note that the two vector fields $X$
  and $Y$ coincide on some neighborhood of the periodic orbit.
  Let $x$ be any point of the periodic orbit, and define $\gamma = W^s(x)$, the
  stable set of $x$ for $X$. According to section~\ref{ss:definitions} we
  have $\Phi^{r_i}_X(W^s(x)) \subset W^s(\Phi_X^{r_i}(W^s(x)) = W^s(x)$.
  Thus point (3) holds. All other properties are obvious.
\end{proof}

\begin{lemma}\label{lem:ModeleLocal}
  Let $r>0$. All the continuous flows belonging to $\cA^+(r)$ are conjugate,
  through a conjugating map which is a homeomorphism of $\mathbb{A}$ isotopic to
  the identity.
\end{lemma}
\begin{proof}
  Let us consider two continuous flows
  $(\Phi_{1}^t)_{t \in \R}$, $(\Phi_{2}^t)_{t \in \R}$ which are elements
  of $\cA^+(r)$, and denote $O_{1}, \gamma_{1}$, $O_{2}, \gamma_{2}$ the
  corresponding objects.
  
  The restriction of $\Phi_{i}^r$ to the arc $\gamma_{i}([-1, 1])$ has a unique
  attracting fixed point $\gamma_{i}(0)$. Hence there exists a homeomorphism
  $h: \gamma_{1}([-1, 1]) \to \gamma_{2}([-1, 1])$ which is positively equivariant,
  meaning that for every $k \geq 0$, $h \circ \Phi_{1}^{kr}  = \Phi_{2}^{kr} \circ h$.
  
  We extend the map $h$ to a homeomorphism of the interior of the annulus,
  still denoted by $h$, by putting
  $h \circ \Phi_{1}^{t}(x)  = \Phi_{2}^{t} \circ h(x)$ for every
  $x \in \gamma_{1}([-1,1])$ and every $t$.
  This is well defined, because of the positive equivariance above.
  This map $h$ now satisfies the conjugacy relation
  $h \circ \Phi_{1}^{t}  = \Phi_{2}^{t} \circ h$ for every~$t$.
  
  We further extend $h$ to the boundary of the annulus by putting
  $h(\alpha(x)) = \alpha(h(x))$ for every non periodic point $x$.
  This completes the construction of a conjugacy between both flows.
  The homeomorphism preserves each boundary component and preserves the
  orientation on the boundary, thus it is homotopic to the identity.
\end{proof}

To complete the proof of the theorem, we will use this local model in combination
with the following observation.

\begin{lemma}\label{lem:ModeleLocalRot}
  Let $\tau: [-1, 1] \to \R$ be a $C^1$ function that vanishes on a neighborhood
  of $-1$ and $1$, and such that $\tau$ takes a given value $1/r$ on a non
  trivial interval.
  Let $T$ be the vector field on the annulus $\mathbb{A}= \mathbb{S}^1 \times [-1,1]$
  defined by $T(x,y) = (\tau(y), 0)$.
  Then the time one map $\Phi_{T}^1$ is in the closure of the set of time one
  maps of elements of $\cA^+(r)$.
\end{lemma}

\begin{proof}
  Pick any point $y_{0}$ interior to an interval where $\tau$ equals $1/r$.
  Take any $C^1$ function $v \colon [-1,1] \to \R$ which is $0$ at $-1$ and $1$,
  positive on $(-1,y_{0})$ and negative on $(y_{0},1)$, and consider the vector
  field $T_v(x,y) = (\tau(y), v(y))$. The map $v$ can be taken as small as we
  want, and we think of $T_v$ as a small perturbation of $T$.
  
  It remains to check that the flow $(\Phi_{T_v}^t)_{t \in \R}$ is
  in $\mathcal{A}^+(r)$.
  First note that the second coordinates of $t \mapsto \Phi_{T_v}^t(x,y)$ is
  given by the flow of the one-dimensional vector field $v$, which vanishes only
  at $y_{0}$.
  Thus the circle $O = \mathbb{S}^1 \times \{y_{0}\}$ is the only periodic orbit
  in the interior of the annulus, and points (2) and first part of point (4) follow.
  Note that, since the function $\tau$ is constant near $y_{0}$, the vertical
  foliation is invariant near $y =y_{0}$ under the flow of $T_v$.
  In particular any small vertical segment crossing $O$ is positively invariant
  under $\Phi_{T_v}^r$, which provides point (3).
  Since $\tau$ vanishes near $-1$ and $1$, the vector field $T_v$ is vertical
  near the boundary, which entails the second part of point (4) and point (5).
  Point (1) is obvious.
\end{proof}

Of course, by reversing the time, all we have said about the attracting periodic
orbits is valid, upon obvious changes, for repelling periodic orbits.
Thus, in the end of the proof below, we will use the statements above for both
types of periodic orbits, denoting $\cA^-(r)$ the repelling counterpart of $\cA^+(r)$.

\subsection{Deletions of periodic orbits}\label{ssec:Deletions}

We are ready to complete the proof of Theorem~\ref{thm:t1flow}.
We keep all the notation from the two preceding paragraphs, and introduce
the following one.
Given numbers $r_1, \dots, r_n$, we denote $\cA{(r_1, \dots , r_n)}$
the set of continuous flows on $\Sigma$ which are the identity outside the $A_i'$'s,
and whose restriction to $A_i'$ is conjugate to an element of $\cA^\pm(r_i)$,
where the sign is chosen according to the dynamics of $Y$, depending whether the
periodic orbit $O_i$ is attracting or repelling.
We will use the same (abusive) notation for the set of vector fields whose flow
is in $\cA{(r_1, \dots , r_n)}$.
Note that any two flows $\Phi, \Psi$ in $\cA{(r_1, \dots , r_n)}$ are conjugate.
Indeed, by Lemma~\ref{lem:ModeleLocal}, for each $i$, we can pick a
homeomorphism of $A_i'$ which conjugates the restriction of $\Phi$ and $\Psi$
and is isotopic to the identity.
Since $\Homeo_{+}(S^1)$ is path-connected, there exists a homeomorphism $h$
of $\Sigma$ that extends the collection of these homeomorphisms, and since the
flows are the identity outside the annuli, this map $h$ conjugates their time
one maps on the entire surface.

Now let $(r_1, ..., r_n)$ denote the periods of the periodic orbits of $Y$ on
each annulus. Note that the vector field $Y$ belongs to the
set $\cA{(r_1, \dots , r_n)}$, according to Lemma~\ref{lem:MSFitsModel}.
Fix $\varepsilon >0$, and consider a vector field $T$ whose restriction to each
annulus $A_i'$ is as in Lemma~\ref{lem:ModeleLocalRot}, with the function $\tau$
taking both values $r_i$ and $\varepsilon$ on non trivial intervals.

A first application of Lemma~\ref{lem:ModeleLocalRot} shows that there is a vector
field $Y'$ in $\cA{(r_1, \dots , r_n)}$ such that the time one map of $T$ is a
limit of conjugates of the time one map of $Y'$. By the above remark, the time
one maps of $Y$ and $Y'$ are conjugate, thus the time one map of $T$ is also a
limit of conjugates of the time one map of $Y$.

A second application of Lemma~\ref{lem:ModeleLocalRot} shows that there is a
vector field $Y_\varepsilon$ in $\cA{(\varepsilon, \dots , \varepsilon)}$ such
that the time one map of $T$ is a limit of conjugates of the time one map
of $Y_\varepsilon$.
On the other hand, one can easily find a vector field
in $\cA{(\varepsilon, \dots , \varepsilon)}$ whose norm in suitable coordinates
is everywhere less than $\varepsilon$.

Putting all this together, we see that the time one map of $Y$ is weakly conjugate
to the time one map of a vector field which is arbitrarily close to $0$.
Finally by applying Lemma~\ref{lem:Flow-Step1} we get that the time one map
of $X$ is weakly conjugate to the identity.

\section{Rotation set versus translation length in the fine graph}\label{sec:RotVSLength}

In this section we consider the case when $\Sigma = \T^2$ is the 2-torus, and we
focus on the group $G=\Homeo_0(\T^2)$ of homeomorphisms isotopic to the identity.
Our objective now is to prove Theorem~\ref{thm:RotVSLength}.

We first recall the definition of the rotation set. Let $f \in \Homeo_0(\T^2)$.
We fix a homeomorphism of the plane $\tilde f$ which is a lift of $f$.
Consider the map
\[
\begin{array}{rcl}
D(\tilde f)\colon\T^2 & \longrightarrow & \mathbb{R}^2 \\
x & \longmapsto & {\tilde f}(\tilde x) - \tilde x
\end{array}
\]
where $\tilde x$ denotes any point of the plane which is a lift of $x$
(the resulting value does not depend on the choice of $\tilde x$).
The rotation set $\rho(\tilde f)$ is the set of vectors $v \in \mathbb{R}^2$
for which there exist a sequence of points $(x_k)_{k \geq 0}$ of the torus,
and a sequence of integers $(n_k)_{k \geq 0}$ that tends to $+\infty$, such
that the sequence
\[  \frac{1}{n_k} D(\tilde f^{n_k})(x_k)  \]
converges to $v$.
This is a non empty compact subset of the plane, and
Misiurewicz and Ziemian proved that the rotation set is also convex (see~\cite{MZ1}),
and that it is also the Hausdorff limit of the sequence $\frac{1}{n}\tilde{f}^n(P)$,
for any fundamental domain $P$ of the action of $\Z^2$ on $\R^2$.
Replacing $\tilde f$ with another lift of $f$ only results in an integer translation
of the rotation set, thus we can define the rotation set $\rho(f)$ as a compact convex
subset of the plane, well defined up to integer translations.

A point $x \in \T^2$ is said to have rotation vector $v\in \R^2$ for the
lift $\tilde f$
if the sequence $D(\tilde f^n)(x)/n$ converges to
$v$ when $n$ tends to infinity.
The rotation set always contains all such rotation vectors.
Furthermore its extreme points
are known to be realized as rotation vectors of points.
Finally, as a consequence of the above, for all $f\in\Homeo_0(\T^2)$ and all
$n\geqslant 1$ we have $\rho(f^n)=n\rho(f)$.

\subsection{Continuity of the rotation set}\label{ssec:RotContinu}

We endow the set of non empty compact subsets of the plane with the Hausdorff
topology, and we denote by $\mathrm{Conv}(\mathbb{R}^2)$ the subspace of convex subsets.
Misiurewicz and Ziemian proved the following theorem
(see \cite[Theorem~2.10]{MZ1} and \cite[Theorem~B]{MZ2}).

\begin{theorem}[Misiurewicz-Ziemian]\label{thm:MZ1}
  The map
  \[ \begin{array}{rcl} \rho\colon\Homeo_0(\T^2) & \longrightarrow &
  \mathrm{Conv}(\mathbb{R}^2)/\Z^2 \\ f & \longmapsto & \rho(f)
  \end{array} \]
  is upper semi-continuous: for every $f_0 \in \Homeo_0(\T^2)$, for every open
  subset $O$ of the plane containing $\rho(f_0)$, there exists a neighborhood $V$
  in $\Homeo_0(\T^2)$ such that every $f \in V$ has its rotation set included in $O$.
  Moreover, $\rho$ is continuous at every point $f_0$ such that $\rho(f_0)$ has
  non empty interior.
\end{theorem}

Note that as a consequence of the first point, the map is also continuous at every
point $f_0$ whose rotation set is a singleton.
It is well-known that this map is not continuous: for instance it is
easy to construct an example of a map $f$ whose rotation set is a horizontal segment,
admitting arbitrarily small deformations $f_\varepsilon$ whose rotation set is
the singleton $\{0\}$, as we did in Paragraph~\ref{ssec:Deletions}.

Here is a consequence of Theorem~\ref{thm:MZ1}.
Denote $\mathrm{Conv}_0(\mathbb{R}^2)$ the subspace of convex sets with empty
interior. This is a closed subset of $\mathrm{Conv}(\mathbb{R}^2)$.
Thus the quotient $Y = \mathrm{Conv}(\mathbb{R}^2) / \mathrm{Conv}_0(\mathbb{R}^2)$
obtained by identifying $\mathrm{Conv}_0(\mathbb{R}^2)$ to a point is a Hausdorff
space. Denote $p$ the quotient map.
The theorem implies that the map $p \circ \rho: G=\Homeo_0(\T^2) \to Y/\Z^2$ is continuous.
Since $\rho$ is a conjugacy invariant, we get a quotient map defined on the space
of weak conjugacy classes,
\[  \begin{array}{rcl} \bar \rho: G/\!/G & \longrightarrow & Y/\Z^2 \\
  \left[ f \right]\ & \longmapsto & [\rho(f)].  \end{array}  \]
This map $\bar \rho$ is continuous. Its image is dense in $Y$, since every rational
polygon is the rotation set of some homeomorphism (see Kwapisz \cite{K}).
Prior to the works~\cite{BHW, BHMMW} and the remarks below, one could have expected
this map to be a homeomorphism.
However, Theorem~\ref{thm:RotVSLength}~(2), that we will prove now, implies that
this map is not injective, as the maps $f_n$ and $g_n$ of Section~\ref{ssec:LenVSRot}
have the same rotation set with nonempty interior, but distinct translation length,
and hence are not weakly conjugate.

\subsection{Independance}\label{ssec:Indep}

We are ready to prove the second part of Theorem~\ref{thm:RotVSLength}.

\begin{proof}[Proof of Theorem~\ref{thm:RotVSLength}~(2)]
  Fix the integer $n >0$.
  The map $v^n h^n$ fixes each of the four points
  $$
  (0,0), \left(0, \frac{1}{2}\right), \left(\frac{1}{2}, 0\right),
  \left(\frac{1}{2}, \frac{1}{2} \right)
  $$
  and their rotation vectors for the lifted map $V^n H^n$ are respectively 
  $$
  (0,0), (n, 0), (0, n), (n, n).
  $$
  Since the rotation set is convex and contains the rotation vectors of every
  fixed point, we deduce that the set $\rho(V^n H^n)$ contains the square $[0,n]^2$.
  
  Conversely, for all $(x,y)\in\R^2$, one can check that
  $V^nH^n(x,y)-(x,y)\in[0,n]^2$, and it follows that $\rho(V^nH^n)$ is contained
  in $[0,n]^2$.
  The arguments for the map $(VH)^n$ are entirely similar.

  Now let us turn to the estimation of the asymptotic translation length.
  Since the rotation set of $(vh)^n$ and $v^n h^n$
  has non empty interior,
  the asymptotic translation lengths of both maps are positive, according to~\cite{BHMMW}.
  The relation $|(vh)^n| = n |vh|$ is immediate from the definition. It remains
  to see that $|v^n h^n|$ is bounded by~$2$.
  Consider curves $\alpha$ and $\beta$ that are projections in the
  torus of some horizontal and vertical lines of the plane, respectively.
  We note that $h^n(\alpha ) = \alpha$, and thus $v^n h^n (\alpha) = v^n (\alpha)$
  meets $\beta$ only once. Thus by triangle inequality we have
  \[
  d(\alpha, v^n h^n(\alpha) ) \leqslant
  d(\alpha, \beta ) + d(\beta, v^n h^n(\alpha) ) = 2.  \qedhere
  \]
\end{proof}

\subsection{Width measurements}
Our next objective is to prove the point~(1) of Theorem~\ref{thm:RotVSLength},
which relates the translation length with the essential width of the rotation set.
Remember that we defined the \emph{essential width} $EW(C)$ of a compact convex
subset $C$ of the plane to be the infimum of the horizontal width of $A(C)$ as $A$
ranges over $\SL_2(\Z)$.
In this section we collect some properties of the essential width; in particular,
we will relate $EW(C)$ to the property that $C$ contains at least three non
aligned integer points.

It is clear from the definition that the essential width is invariant by translations and under change of basis of $\Z^2$.
It is, obviously, increasing, with respect to inclusion.
Furthermore, since linear homotheties commute with all linear maps, it is homogeneous:
for every convex compact $C\subset\R^2$ and every homothety~$h$
of ratio~$r$, we have $EW(h(C))=r EW(C)$. 
This implies that the essential width
of the rotation set behaves multiplicatively upon taking powers of homeomorphisms.

Let $C$ be a convex compact subset of the plane.
If $C$ contains three non aligned integer points, then
so does $A(C)$ for any $A\in\SL_2(\Z)$, hence $EW(C) \geqslant 1$.
More generally, if $C$ has non empty interior, then for $r$ large enough $rC$
contains three non aligned integer points, and by homogeneity, we deduce
that $EW(C)>0$.
The converse implication, namely that if $C$ has empty interior then $EW(C)=0$,
follows easily from Dirichlet approximation theorem (as already noted in the
proof of~\cite[Theorem~1.3]{BHMMW}). Proposition~\ref{prop:CompareWidth} below
may be viewed as a quantitative version of this equivalence.

We will now prove that the map $EW$ is continuous.
Remember that $\mathrm{Conv}(\mathbb{R}^2)$ denotes the set of non empty,
compact, convex subsets of the plane, endowed with the Hausdorff topology.
\begin{lemma}\label{lem:EW-Cont}
  The map $EW\colon \mathrm{Conv}(\mathbb{R}^2) \to\R_{\geqslant 0}$
  is continuous.
\end{lemma}
\begin{proof}
  This map $EW$ is defined as the infimum of a family of continuous functions.
  Hence, it is upper semicontinuous, and continuous at every point where it
  equals~$0$, \textsl{i.e.}, the convex sets with empty interior. It remains
  to prove the continuity of $EW$ at a point $C$ of 
  $\mathrm{Conv}(\mathbb{R}^2)$ with non empty interior.
  Consider a point
  $x_0$ in the interior of $C$, and denote by $h_r$ the homothety of center $x_0$
  and ratio $r>0$. A basis of neighborhoods of $C$ for the Hausdorff topology
  consists in the sets $V_\varepsilon$, $\varepsilon>0$, where $V_\varepsilon$ is
  the set of convex sets containing $h_{1-\varepsilon}(C)$ and contained
  in $h_{1+\varepsilon}(C)$.
  By the homogeneity and monotonicity properties recalled above, the restriction
  of $EW$ to $V_\varepsilon$
  takes values in the interval $[(1-\varepsilon)EW(C),(1+\varepsilon)EW(C)]$.
  It follows that $EW$ is continuous at~$C$.
\end{proof}

We now relate the
essential width of a compact convex set $C$ with the property that the interior
of $C$ contains three non aligned integer points.

\begin{proposition}\label{prop:CompareWidth}
  Let $C$ be a convex compact subset of $\R^2$.
  \begin{enumerate}
  \item If the interior of $C$ contains three non aligned integer points,
    then $EW(C) > 1$.
  \item If $EW(C) > 4$, then the interior of $C$ contains three non aligned
    integer points.
  \end{enumerate}
\end{proposition}

The first point is sharp, but the second is not.
A relevant example is the convex hull of the triangle with
vertices $(-1,0), (2/3, 5/3), (7/3, -5/3)$.
Its interior contains two integer points, namely $(0,0)$ and $(1,0)$. Its essential
width is $10/3$.
We suspect this is the maximal essential width of a (unique, essentially)
convex compact set not containing three non aligned points
in its interior. However, we did not find a proof that would not run into a
tedious case by case discussion, and decided to provide instead a simple proof with the
slightly worse bound~$4$.
We did not find any appropriate reference either,
even though the relation between the size of convex sets, often measured in terms of
area, and the number of integer interior points, has a long history,
see e.g.~\cite{Ehrhart}.

\begin{proof}
  For the first point, assume the interior of $C$ contains three non aligned  
  integer points. We may choose $0< r< 1$ such that $rC$ still contains the same
  three points. As noted above we have $EW(rC) \geq 1$, and by homogeneity 
  we get $EW(C) >1$.

  Let us prove the second point. We argue by contradiction.
  Let $C$ be a convex compact subset of the plane such that $EW(C)>4$ and assume
  that the interior of $C$ does not contain three non aligned integer points.
  
  We first consider the case when the interior of $C$ contains two integer points
  and find a contradiction in that case, then we will deduce the general case.
  By convexity the segment between those two points is entirely included in the
  interior of $C$. Thus $C$ contains two consecutive integer points.
  Hence up to a change of basis, and integer translation
  we may suppose that the interior of $C$ contains the points $(0,0)$ and $(1,0)$.
  We choose a point $(x_1,y_1)$ of $C$ maximizing $y_1$, and a point $(x_2,y_2)$
  of $C$ minimizing $y_2$. Thus $y_1>0>y_2$, and up to reversing the sign of the
  second coordinate we will also suppose that $y_1\geqslant |y_2|$. Since $EW(C)>4$,
  we have $y_1+|y_2|>4$, hence $y_1>2$. It follows that $C$ intersects
  one of the intervals $(n,n+1)\times\{1\}$, and in fact it intersects only one
  of these, otherwise the interior of $C$ would contain an integer point non aligned
  with $(0,0)$ and $(1,0)$.
  Up to replacing $C$ with $T^{-n}C$, where
  $T=\left(\begin{array}{cc}1 & 1 \\ 0 & 1\end{array}\right)$ is the horizontal
  transvection, we may further suppose that $n=0$.
  It follows that $x_1\in[0,1]$.
  
  By using the convexity
  of $C$ and the hypothesis on the integer points, we see that the intersection
  of $C$ with the half plane $\{y\geqslant 1\}$ is contained in the strip
  $\{0\leqslant x\leqslant 1\}$, and the intersection of $C$ with the half plane
  $\{y\leqslant 1\}$ is contained in the triangle with sides included in the
  lines $((x_1,y_1),(0,1))$, $((x_1,y_1),(1,1))$ and $\{y=y_2\}$.
  
  Let $B$ denote the length of the base, horizontal side of this triangle.
  Since $EW(C)>4$, and since both the triangle and the strip mentioned above
  have the same horizontal projection as $B$ itself,
  we have $B>4$. By the intercept theorem, we have
  $\frac{y_1-1}{1} = \frac{y_1+|y_2|}{B}$, thus $y_1-1<\frac{y_1+|y_2|}{4}$.
  As $y_1\geqslant |y_2|$, we deduce $y_1-1<\frac{y_1}{2}$, hence $y_1<2$:
  this is a contradiction.
  
  It remains to deal with the case when the interior of $C$ does not contain two 
  integer points. The strategy in this case is to make $C$ grow until its interior 
  contains exactly two integer points, and then apply the previous case. Note that 
  the essential width is increasing with 
  respect to inclusion, so the only difficulty is to show that we can make $C$ grow 
  without eating too many integer points at the same time.
  
  If $C$ has no integer points in its boundary, and if $p\in\Z^2\smallsetminus C$
  is an integer point closest to $C$, then the convex hull
  $\mathrm{Hull}(C\cup B_\varepsilon(p))$ of the union of $C$ with the ball of
  center $p$ and radius $\varepsilon$, for all $\varepsilon>0$ small enough,
  has in its interior exactly one integer point, $p$, more than $C$ has. Also,
  it still has no integer points in its boundary, so we may proceed if necessary.
  
  In the general case, when $C$ may have integer points in its boundary, we may
  consider a homothety $h$ with center in $C$ and radius $r<1$,
  and set $C'=h(C)$.
  For $r$ close enough to $1$ we still have $EW(C')>4$, since $EW(C)$ depends
  continuously on $C$ by Lemma~\ref{lem:EW-Cont}, also, $C'$ has no integer points
  at its fronteer, and has the same number of integer points as $C$ in its interior,
  so we may use the preceding case.
\end{proof}

\subsection{Some relations}\label{ssec:Relations}

Now we are ready to prove the point (1) of Theorem~\ref{thm:RotVSLength}.

Given a subset $R$ of the plane, denote by $|R|$ the infimum of the asymptotic
translation length of those torus homeomorphisms $f$ that are isotopic to the
identity and whose rotation sets contain $R$ in its interior.
We denote $T_0=\{ (0,0), (0,1), (1,0) \}$.

\begin{lemma}\label{lem:T0}
  We have $|T_0| >\frac{1}{222}$.
\end{lemma}

\begin{proof}
  This proof consists essentially
  in following the proof of Theorem 1.3 of \cite{BHMMW}, while keeping track of
  estimates whenever possible. Let $f$ be a homeomorphism of the torus, isotopic to 
  the identity, whose rotation set contains $T_0$ in its interior.
  By Franks' theorem \cite{Franks},
  we may find three distinct points $A,B,C$ that are fixed by $f$ and whose
  rotation vectors respectively equal the three integer elements of $T_0$.
  By Llibre-McKay's argument (part 1 of the proof of the main theorem, p122 of \cite{LM}),
  the isotopy class of the map $f$ relative to $\{A,B,C\}$ is pseudo-Anosov.
  By \cite[Lemma~4.2]{BHW}, the asymptotic translation length $|f|$ is not
  smaller than the asymptotic translation length of the action of $f$  on the
  (usual) curve graph of the surface $S = \T^2 \setminus \{A,B,C\}$.
  Note that the topology of $S$ does not depend on the points $A,B,C$, that is,
  any two copies of the torus minus three points are homeomorphic.
  
  Finally by the work of Masur and Minsky \cite[Proposition 4.6]{MM99},
  the asymptotic translation length
  of pseudo-Anosov maps acting on the curve graph of a given non sporadic finite
  type surface $S$ is bounded from below by a positive constant $\|S\|$ that
  depends only on the topology of the surface.
  More precisely, on the 3-punctured torus, a theorem of Gadre and Tsai gives
  the bound 
  $\frac{1}{222}$ (\cite[Theorem~5.1]{GadreTsai}). This completes the proof.
\end{proof}

\begin{proof}[Proof of point (1) of Theorem~\ref{thm:RotVSLength}]
  We first consider the case when $c=1$. Let $f\in \Homeo_0(\T^2)$ be such that
  $0<EW(\rho(f))\leqslant 1$. Let $n$ be the smallest integer so that
  $nEW(\rho(f)) > 4$.
  We have $(n-1)EW(\rho(f)) \leqslant 4$, so
  $n \leqslant \frac{4}{EW(\rho(f))} + 1 \leqslant \frac{5}{EW(\rho(f))}$. 

  By Proposition~\ref{prop:CompareWidth}, the interior of $\rho(f^n)$ contains three
  non aligned integer points, and by Lemma~\ref{lem:T0} and $\SL_2(\Z)$-equivariance
  we have $|f^n|>\frac{1}{222}$. Hence
  
  \[ |f| > \frac{1}{222n}\geqslant \frac{EW(\rho(f))}{5 \times 222} =
  \frac{EW(\rho(f))}{1110}. \]
  
  This proves that the constant $m_1=\frac{1}{1110}$ suits our need.
  This is obviously not sharp; it would seem quite challenging, though, to find
  sharp constants. Obviously we can take $m_c = m_1$ for every $c \leqslant 1$.
  
  Finally, suppose $c > 1$ and let $f$ be such that $EW(\rho(f)) \leqslant c$.
  Assuming additionaly that $EW(\rho(f)) >1$  we get that $EW(\rho(f^4)) >4$ and,
  as before $|f^4| > |T_0|$, so that
  \[ |f|\geqslant \frac{|T_0|}{4}\geqslant \frac{|T_0|}{4c}EW(\rho(f)) \geqslant 
  \frac{1}{888c}EW(\rho(f)), \]
  so by putting this together with our first case we see that the
  constant $m_c = \frac{1}{\max(888c, 1110)}$ suits our need.
\end{proof}

\begin{proof}[Proof of Corollary~\ref{cor:NoBloodyRoots}]
  Denote $G=\Homeo_0(\T^2)$.
  By Proposition~\ref{thm:MZ1}, Lemma~\ref{lem:EW-Cont}, and the continuity of
  the asymptotic translation length~\cite[Theorem 3.4]{BHMMW},
  every $f\in G$ with rotation set having non
  empty interior is a point of continuity of the map
  $r\colon g\mapsto \frac{|g|}{EW(\rho(g))}$.
  By Theorem~\ref{thm:RotVSLength}~(2), the image of this ratio map $r$
  has $0$ as an accumulation point.
  Consider the set
  \[ U_0 = \left\lbrace [h]\in G\dq G \, , \, r(h)<\frac{m_1}{2} \right\rbrace, \]
  where $m_1$ is the lowest possible bound with the notation of
  Theorem~\ref{thm:RotVSLength}~(1).
  Then $U_0$ is nonempty and the continuous map
  $[h]\mapsto EW(\rho(h))$ is bounded below by $1$ on $U_0$. Let $ew_0$ be the
  infimum of $EW$ on this set. Let $\varepsilon>1$. Now consider the set
  \[ U_1 =
  \left\lbrace [h]\in U_0 \, , \, EW(\rho(h))\in(ew_0, ew_0\varepsilon) \right\rbrace.\]
  This is nonempty and open. Let $h$ such that $[h]\in U_1$, and let $f$ be a
  $\frac{p}{q}$-root of $h$. Since $f^p = h^q$, we get that
  $EW(\rho(h)) = \frac{p}{q} EW(\rho(f))$, and $r(f) = r(h)$.
  We first deduce that $[f] \in U_{0}$, thus $EW(\rho(f)) \geqslant ew_{0}$,
  and $ew_0 \varepsilon > EW(\rho(h)) \geqslant \frac{p}{q} ew_{0}$.
  We conclude that $\frac{p}{q} < \varepsilon$.
  This proves that $h$ has no $\frac{p}{q}$-root with $\frac{p}{q} \geqslant \varepsilon$.
\end{proof}
Of course, the same argument provides explicit maps without $n$th roots for
explicit (large) $n$. For example, if we consider the map $f_N=v^Nh^N$ from
Section~\ref{ssec:LenVSRot} with $N$ larger than $\frac{2}{m_1}$,
then $EW(\rho(f_N)) = N$ and $r(f_N) < m_{1}$; we deduce that for all
$\frac{p}{q}>N$, the map $f_N$ is without $\frac{p}{q}$-th roots, and shares
this property with all maps $h$ such that $[h]$ is in a sufficiently small
neighborhood of $[f_N]$ in $G\dq G$.


\bibliographystyle{plain}

\bibliography{BibWC}

\end{document}